\documentclass[12pt,reqno]{amsart}

\usepackage{amsthm, amsmath, amsfonts, amssymb, graphicx, tikz-cd, makecell, hyperref, mathtools, mathrsfs, comment,needspace}
\DeclareMathAlphabet{\mathbbb}{U}{bbold}{m}{n}
\usepackage{thmtools}
\usepackage[capitalize,noabbrev]{cleveref} 
\usetikzlibrary{arrows.meta,calc}

\newtheorem{theorem}{Theorem}[section]
\newtheorem{lemma}[theorem]{Lemma}
\newtheorem{proposition}[theorem]{Proposition}

\theoremstyle{definition}
\newtheorem{definition}[theorem]{Definition}
\newtheorem{example}[theorem]{Example}

\newtheorem{remark}[theorem]{Remark}
\newtheorem{claim}[theorem]{Claim}

\newenvironment{proofclaim}{\paragraph{\emph{Proof of the Claim}.}}{\hfill$\qed$\\}

\newcounter{mycount}

\newcommand{\myref}[1]{\hyperref[#1]{#1}}

\newcommand{\sfour}{\mathsf{S4}}
\newcommand{\ipc}{\mathsf{IPC}}

\newcommand{\M}{\mathsf{M}}

\newcommand{\K}{\mathbf{K}}

\newcommand{\R}{\mathscr{R}}
\renewcommand{\H}{\mathfrak{H}}

\renewcommand{\L}{\mathsf{L}}
\newcommand{\F}{\mathfrak{F}}
\newcommand{\G}{\mathfrak{G}}

\newcommand{\mipc}{\mathsf{MIPC}}
\newcommand{\mgrz}{\mathsf{MGrz}}
\newcommand{\mpipc}{\mathsf{M^+IPC}}
\newcommand{\mpgrz}{\mathsf{M^+Grz}}
\newcommand{\msfour}{\mathsf{MS4}}
\newcommand{\mpsfour}{\mathsf{M^+S4}}
\newcommand{\sfive}{\mathsf{S5}}

\newcommand{\ms}{\mathsf{MS4}}

\newcommand{\cpc}{\mathsf{CPC}}

\renewcommand{\L}{\mathsf{L}}
\newcommand{\Lae}{\mathcal{L_{\forall\exists}}}
\newcommand{\Lba}{\mathcal{L_{\Box\forall}}}

\newcommand{\iqc}{\mathsf{IQC}}
\newcommand{\sk}{\rho \hspace{0.08em}}

\newcommand{\msfrm}{{\mathbf{DF}_{\ms}}}
\newcommand{\mipcfrm}{{\mathbf{DF}_{\mipc}}}
\newcommand{\Grz}{\mathsf{Grz}}
\newcommand{\grz}{\mathsf{grz}}

\newcommand{\DF}{\mathbf{DF}}
\newcommand{\Rrel}{\mathrel{R}}
\newcommand{\Qrel}{\mathrel{Q}}
\newcommand{\Erel}{\mathrel{E}}
\newcommand{\EQrel}{\mathrel{E_Q}}
\newcommand{\Fin}{\mathbf{Fin}}
\newcommand{\mipcfin}{{\Fin_\mipc}}
\newcommand{\mgrzfin}{{\Fin_\mgrz}}

\newcommand{\mpipcfin}{{\Fin_\mpipc}}
\newcommand{\mpgrzfin}{{\Fin_\mpgrz}}

\newcommand\clusterone[2]{
  \path[draw,red] let \p1=(#1)
    in \pgfextra{
    \pgfmathsetmacro{\radius}{#2*0.3}
  }
  (\p1) circle(\radius cm);
}

\newcommand\clustertwo[4]{
  \path[draw,red] let \p1=(#1), \p2=(#2), \p3=($(\p1)!.5!(\p2)$)
  in \pgfextra{
    \pgfmathsetmacro{\angle}{atan2(\y2-\y1,\x2-\x1)}
    \pgfmathsetmacro{\focal}{veclen(\x2-\x1,\y2-\y1)/2/1cm}
    \pgfmathsetmacro{\lentotcm}{\focal*2*#3}
    \pgfmathsetmacro{\axeone}{(\lentotcm - 2 * \focal)/2+\focal}
    \pgfmathsetmacro{\axetwo}{sqrt((\lentotcm/2)*(\lentotcm/2)-\focal*\focal}
    \pgfmathsetmacro{\newaxetwo}{\axetwo*0.5*#4}
  }
  (\p3) ellipse[x radius=\axeone cm,y radius=\newaxetwo cm, rotate=\angle];
}

\makeatletter 
\edef\plabelformat{(\string#2\string#1\string#3)}
\edef\plabelrangeformat{(\string#3\string#1,\string#2\string#6)}
\newcommand{\plabel}[1]{\label{#1}
\immediate\write\@auxout{\noexpand\crefformat{#1}{\noexpand\cref{#1}\plabelformat}
\noexpand\crefmultiformat{#1}{\noexpand\cref{#1}\plabelformat}{,\plabelformat}{,\plabelformat}{,\plabelformat}
\noexpand\crefrangeformat{#1}{\noexpand\cref{#1}\plabelrangeformat}}}
\makeatother

\makeatletter
\@namedef{subjclassname@2020}{\textup{2020} Mathematics Subject Classification}
\makeatother

\pgfdeclarelayer{edgelayer}
\pgfdeclarelayer{nodelayer}
\pgfsetlayers{edgelayer,nodelayer,main}

\tikzstyle{none}=[inner sep=0pt]
\tikzstyle{black dot}=[fill=black, draw=black, shape=circle, inner sep=0, minimum size=3.5pt]

\tikzstyle{to}=[->]
\tikzstyle{mapsto}=[{|->}]
\tikzstyle{none dashed}=[-, dashed]
\tikzstyle{dashed to}=[->, dashed]
\tikzstyle{dashed mapsto}=[{|->}, dashed]
\tikzstyle{Latex arrow}=[{-{Latex[width=1mm]}}]
\tikzstyle{dashed Latex arrow}=[{-{Latex[width=1mm]}}, dashed]

\begin{document}

\title{Esakia's theorem for the amended monadic intuitionistic calculus}

\author{G.~Bezhanishvili}
\address{New Mexico State University}
\email{guram@nmsu.edu}

\author{L.~Carai}
\address{University of Milan}
\email{luca.carai.uni@gmail.com}


\begin{abstract}
We show that the amended monadic Grzegorczyk logic $\mpgrz$ is the largest modal companion of the amended monadic intuitionistic logic $\mpipc$. Thus, unlike the monadic intuitionistic logic $\mipc$, Esakia's theorem does extend to $\mpipc$.
\end{abstract}

\maketitle
\tableofcontents

\section{Introduction}

It is a classic result that the Grzegorczyk logic $\Grz$ is the largest modal companion of the intuitionistic propositional calculus $\ipc$ (see \cite{Esa79}). In \cite{Nau91} it was claimed that Esakia's theorem does not extend to the predicate setting. While the proof contains a gap, it is indeed the case that the monadic intuitionistic calculus $\mipc$ has no largest modal companion (see \cite{BC24a}). 
Our aim is to show that Esakia's theorem does hold for the amended calculus $\mpipc$. The latter is obtained by postulating the monadic version of Casari's axiom 
\[
\forall x[(p(x) \to \forall x p(x)) \to \forall x p(x)] \to \forall x p(x),
\]
and we prove that $\mpgrz$ is the largest modal companion of $\mpipc$, where $\mpgrz$ is the amendment of the monadic Grzegorczyk logic $\mgrz$ with the G\"odel translation of the monadic Casari axiom.

We briefly describe the methodology of proving Esakia's theorem for $\ipc$, why it fails for $\mipc$, and why things improve for $\mpipc$. 
Associating with each descriptive $\sfour$-frame its skeleton defines a functor $\rho$ from the category of descriptive $\sfour$-frames 
to the category of descriptive $\ipc$-frames (Esakia spaces).
This functor has a right adjoint $\sigma$ and the two functors yield an equivalence between the categories of finite $\ipc$-frames (finite posets and p-morphisms) and finite $\Grz$-frames. 
Together with the finite model property (fmp for short) of $\ipc$, 
this gives that $\Grz$ is a modal companion of $\ipc$. To see that it is the largest such, let $\M$ be a modal companion of $\ipc$. Since $\Grz$ also has the fmp, 
it is enough to observe that each finite $\Grz$-frame $\F$ is an $\M$-frame. 
Consider $\rho\F$. 
Because $\M$ is a modal companion of $\ipc$, there is a descriptive $\M$-frame $\G$ such that $\rho\F\cong\rho\G$, so $\sigma\rho\F\cong\sigma\rho\G$. 
Since $\sigma\rho\G$ is a p-morphic image of $\G$ and $\F\cong\sigma\rho\F\cong\sigma\rho\G$, we conclude that $\F$ is a p-morphic image of 
$\G$. 
Thus, $\F$ is an $\M$-frame.

Things don't go so smoothly for $\mipc$. On the positive side, 
both $\mipc$ and $\mgrz$ do have the fmp (although the proofs are considerably more complicated; see \cite[Sec.~10.3]{GKWZ03} and \cite[Sec.~7]{BK24}), and they share the same finite frames. However, 
 the trouble is that 
$\sigma$ is no longer well defined. Nevertheless, it is well defined on finite $\mipc$-frames and $\sigma\rho\F\cong\F$ for each finite $\mgrz$-frame $\F$ (although $\rho$ and $\sigma$ no longer establish an equivalence between the two categories of finite $\mgrz$-frames and finite $\mipc$-frames 
since the notion of p-morphism differs for $\mgrz$ and $\mipc$; see \cref{sec: mipc,sec: ms4}).
From a characterization of modal companions of monadic intuitionistic logics (see \cite[Thm.~5.12(2)]{BC24a}), there is a p-morphism from $\rho\G$ onto $\rho\F$, but this no longer implies that there is a p-morphism from $\G$ onto $\F$ (in spite of the fact that $\F\cong\sigma\rho\F$). This is at the heart of the failure of Esakia's theorem for $\mipc$ (see \cite{BC24b}). 

The situation improves for $\mpipc$. Indeed, p-morphisms between finite $\mpipc$-frames and finite $\mpgrz$-frames turn out to coincide 
(and hence $\rho$ and $\sigma$ do yield an equivalence between the categories of finite $\mpgrz$-frames and finite $\mpipc$-frames). 
Moreover, both $\mpipc$ and $\mpgrz$ have the fmp (see \cite{BBI23}). Our key observation is that if $\G$ is a descriptive $\M$-frame and $\F$ is a finite $\mgrz$-frame, each p-morphism from $\rho\G$ onto $\rho\F$ lifts to a p-morphism from $\G$ onto $\F\cong\sigma\rho\F$, thus yielding Esakia's theorem for $\mpipc$. 

\section{Monadic intuitionistic logics}\label{sec: mipc}

In this section we briefly recall monadic intuitionistic logics and their descriptive frame semantics. 
Let $\mathcal L$ be the propositional language of $\ipc$, and let $\Lae$ be its extension by two ``quantifier modalities" 
$\forall$ and $\exists$. 

\begin{definition} \label{def: mipc}
    \cite[p.~38]{Pri57}
    The \textit{monadic intuitionistic propositional calculus} $\mipc$ is the smallest set of formulas in the language
$\Lae$ containing
\begin{itemize}
\item all theorems of $\ipc$;
\item the $\sfour$-axioms for  $\forall$: $\forall(p\land q)\leftrightarrow(\forall p\land\forall q)$, $\forall p \rightarrow p$, $\forall p \rightarrow \forall \forall p$; 
\item the $\sfive$-axioms for $\exists$: $\exists(p\vee q)\leftrightarrow(\exists p\vee\exists q)$, $p \rightarrow \exists p$, $\exists \exists p \rightarrow \exists p$, $(\exists p \land \exists q) \rightarrow \exists (\exists p \land q)$;
\item the connecting axioms: 
 $\exists\forall p\leftrightarrow\forall p$,  $\exists p \leftrightarrow \forall\exists p$
\end{itemize}
and closed under the rules of modus ponens, substitution, and $\forall$-necessitation $(\varphi / \forall \varphi )$.
\end{definition}

It is well know that $\mipc$ axiomatizes the monadic fragment of $\iqc$. Indeed, following \cite[Sec.~3]{Ono87}, fix an individual variable $x$, associate with each propositional letter $p$ the monadic predicate $p^*(x)$, and set 
\begin{itemize}
    \item $p^* = p^*(x)$;
    \item $(\neg \varphi)^* = \neg \varphi^*$;
    \item $(\varphi \circ \psi)^* = \varphi^* \circ \psi^*$ where $\circ =  \wedge,\vee,\to$;
    \item $(\forall \varphi)^* = \forall x \varphi^*$ and $(\exists \varphi)^* = \exists x \varphi^*$.
\end{itemize}
Then we have the following result of Bull \cite{Bul66} (see also \cite{OS88}).

\begin{theorem}
    $\mipc\vdash\varphi$ iff $\iqc\vdash\varphi^*$.
\end{theorem}

\begin{definition}
    A {\em monadic intuitionistic logic} is a set of formulas of $\Lae$ containing $\mipc$ and closed under the rules of inference in \cref{def: mipc}.
\end{definition}

 Each monadic intuitionistic logic is complete with respect to its descriptive frame semantics \cite{Bez99}, which we recall next. 
As usual, for a binary relation $R$ on a set $X$ and $U \subseteq X$, we write $R[U]$ for the $R$-image and $R^{-1}[U]$ for the $R$-inverse image of $U$. 
When $U= \{x\}$, we simply write $R[x]$ and $R^{-1}[x]$. We call $U$
an {\em $R$-upset} if $R[U]\subseteq U$ and an {\em $R$-downset} if $R^{-1}[U] \subseteq U$.
If $R$ is a quasi-order (reflexive and transitive relation), we denote by $E_R$ the equivalence relation given by 
\begin{equation}\label{eq: ER}\tag{$\ast$}
x \mathrel{E_R} y \ \Longleftrightarrow \ x \Rrel y \; \& \; y \Rrel x.
\end{equation}

We recall that a {\em Stone space} is a topological space $X$ that is compact, Hausdorff, and zero-dimensional. We call a binary relation $R$ on $X$ \emph{continuous} if $R[x]$ is closed for each $x \in X$ and $R^{-1}[U]$ is clopen for each clopen $U \subseteq X$.
The next definition goes back to \cite[Sec.~4]{Bez99} (see also \cite[Def.~2.7]{BC24a}). 

\begin{definition}\plabel{def:ono}
A \textit{descriptive $\mipc$-frame} is a tuple $\mathfrak F=(X,R,Q)$ such that
\begin{enumerate}
\item\label[def:ono]{def:ono:item1} $X$ is a Stone space,
\item\label[def:ono]{def:ono:item2} $R$ is a continuous partial order,
\item\label[def:ono]{def:ono:item3} $Q$ is a continuous quasi-order,
\item\label[def:ono]{def:ono:item4} 
$U$ a clopen $R$-upset $\Longrightarrow$ $Q[U]$ is a clopen $R$-upset,
\item\label[def:ono]{def:ono:item5} $R \subseteq Q$,
\item\label[def:ono]{def:ono:item6} $x \Qrel y \Longrightarrow \exists z \in X : x \Rrel z \; \& \; z \EQrel y$.
\end{enumerate}
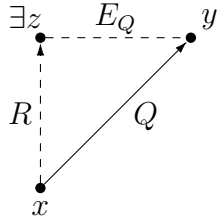
\begin{figure}[!ht]
\begin{center}
\begin{tikzpicture}
	\begin{pgfonlayer}{nodelayer}
		\node [style=black dot] (1) at (0, 0) {};
		\node [style=black dot] (2) at (0, 2) {};
		\node [style=black dot] (3) at (2, 2) {};
		\node [style=none] (4) at (1, 2.25) {$E_Q$};
		\node [style=none] (5) at (-0.27, 1) {$R$};
		\node [style=none] (6) at (1.4, 0.96) {$Q$};
		\node [style=none] (7) at (0, -0.27) {$x$};
		\node [style=none] (8) at (-0.2, 2.27) {$\exists z$};
		\node [style=none] (9) at (2.25, 2.25) {$y$};
	\end{pgfonlayer}
	\begin{pgfonlayer}{edgelayer}
		\draw [style=dashed Latex arrow] (1) to (2);
		\draw [style=Latex arrow] (1) to (3);
		\draw [style=none dashed] (2) to (3);
	\end{pgfonlayer}
\end{tikzpicture}
\end{center}
\end{figure}
\end{definition}

Observe that if a descriptive $\mipc$-frame is finite, then the topology is discrete. More generally, forgetting the topology results in the standard Kripke semantics for $\mipc$ (see, e.g., \cite[Sec.~3]{Ono77}).

We next recall how to interpret the formulas of $\Lae$ in a descriptive $\mipc$-frame $\F=(X,R,Q)$. 
A \textit{valuation} on $\F$ is a map $v$ associating a clopen $R$-upset to each propositional letter. The interpretation of intuitionistic connectives $\wedge,\vee,\to,\neg$ in $\F$ is standard (see, e.g., \cite[pp.~236--237]{CZ97}). To see how $\forall$ and $\exists$ are interpreted, let $x \in X$. Then, for each formula $\varphi$ of $\Lae$,
\begin{align*}
x \vDash_v \forall \varphi & \iff (\forall y \in X) (x \Qrel y \Longrightarrow  y \vDash_v \varphi);\\
x \vDash_v \exists \varphi & \iff (\exists y \in X) (x \EQrel y \; \& \; y \vDash_v \varphi) \\
& \iff (\exists y \in X) (y \Qrel x \; \& \; y \vDash_v \varphi).
\end{align*}

\begin{theorem}
\cite[Thm.~14]{Bez99} Each monadic intuitionistic logic is complete with respect to its class of descriptive $\mipc$-frames.
\end{theorem}

\section{Monadic extensions of \texorpdfstring{$\msfour$}{MS4}}\label{sec: ms4}

In this section we briefly recall monadic extensions of $\msfour$ and their descriptive frame semantics.
Let $\mathcal{L}_{\Box \forall}$ be a propositional modal language with two modalities $\Box$ and $\forall$. 

\begin{definition}\label{def:ms4}
The \emph{monadic $\sfour$}, denoted $\ms$, is the smallest set of formulas of 
$\mathcal{L}_{\Box \forall}$ containing all theorems of the classical propositional calculus $\cpc$, the $\sfour$-axioms for $\Box$, the $\sfive$-axioms for $\forall$, the left commutativity axiom
\[
\Box \forall p \to \forall \Box p,
\]
and closed under the rules of modus ponens, substitution, $\Box$-necessitation, and $\forall$-necessi\-tation.
\end{definition}

As with $\mipc$, we have that $\msfour$ is the monadic fragment of predicate $\sfour$ (see \cite[Thm.~8]{FS77} and \cite[Thm.~5.8]{BK24}).

\begin{definition}
A \emph{monadic extension of $\ms$} is a set of formulas of $\mathcal{L}_{\Box \forall}$ containing $\ms$ and closed under the rules of inference in \cref{def:ms4}.    
\end{definition}

As with monadic intuitionistic logics, each monadic extension of $\msfour$ is complete with respect to its descriptive frame semantics. 

\begin{definition}\plabel{def:descriptive ms-frame}
\cite[Def.~3.7]{BC24a} A \emph{descriptive $\ms$-frame} is a tuple $\mathfrak G=(Y,R,E)$ such that
\begin{enumerate}
\item\label[def:descriptive ms-frame]{def:descriptive ms-frame:item1} $Y$ is a Stone space,
\item\label[def:descriptive ms-frame]{def:descriptive ms-frame:item2} $R$ is a continuous quasi-order,
\item\label[def:descriptive ms-frame]{def:descriptive ms-frame:item3} $E$ is a continuous equivalence relation,
\item\label[def:descriptive ms-frame]{def:descriptive ms-frame:item4}
$x \Erel y \; \& \; y \Rrel z \Longrightarrow \exists u \in Y : x \Rrel u \; \& \; u \Erel z$.
\end{enumerate}
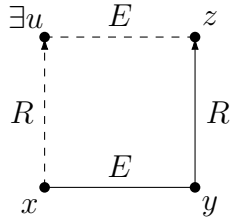
\begin{figure}[!ht]
\begin{center}
\begin{tikzpicture}
\node at (-0.2,-0.25) {$x$};
\node at (-0.25,2.27) {$\exists u$};
\node at (2.2,-0.25) {$y$};
\node at (2.2,2.25) {$z$};
\fill (0,0) circle(2pt);
\fill (0,2) circle(2pt);
\fill (2,0) circle(2pt);
\fill (2,2) circle(2pt);
\draw [dashed, -{Latex[width=1mm]}] (0,0) -- (0,2);
\draw [-{Latex[width=1mm]}] (2,0) -- (2,2);
\draw [dashed] (0,2) -- (2,2);
\draw (0,0) -- (2,0);
\node [above] at (1,2) {$E$};
\node [above] at (1,0) {$E$};
\node [left] at (0,1) {$R$};
\node [right] at (2,1) {$R$};
\end{tikzpicture}
\end{center}
\end{figure}
\end{definition}

As with descriptive $\mipc$-frames, if a descriptive $\ms$-frame is finite, then the topology is discrete. More generally, forgetting the topology results in the standard Kripke semantics for $\ms$ (see, e.g., \cite[Sec.~3]{BC23a}).

We conclude this section by recalling how to interpret the formulas of $\Lba$ in a descriptive $\msfour$-frame $\G=(Y,R,E)$.
A \textit{valuation} on $\G$ is a map $v$ associating a clopen subset of $Y$ to each propositional letter. The interpretation of classical propositional connectives in $\G$ is standard. The modalities $\Box$ and $\forall$ are interpreted as follows, where $x \in Y$ and $\varphi$ is a formula of $\Lba$:
\begin{align*}
x \vDash_v \Box \varphi & \iff (\forall y \in Y) (x \Rrel y \Longrightarrow  y \vDash_v \varphi);\\
x \vDash_v \forall \varphi & \iff (\forall y \in Y) (x \Erel y \Longrightarrow  y \vDash_v \varphi).
\end{align*}

As with monadic intuitionistic logics, we have:

\begin{theorem}
\cite[Thm.~2.24]{BC24b} Each monadic extension of $\msfour$ is complete with respect to its class of descriptive $\msfour$-frames.
\end{theorem}

\section{Modal companions}

In this section we connect modal intuitionistic logics to modal extensions of $\msfour$. We start by recalling that Fischer Servi \cite{FS77} (see also \cite{FS78a}) extended the G\"odel translation 
$(-)^t$ 
of $\ipc$ into $\sf S4$ to a translation of $\mipc$ into $\ms$ by adding the following two clauses for $\forall$ and $\exists$:
\begin{align*}
(\forall \varphi)^t = \Box\forall \varphi^t \quad \mbox{and} & \quad 
(\exists \varphi)^t = \exists \varphi^t. 
\end{align*}

\begin{theorem}\cite{FS77,FS78a}\label{thm: MS4 modal comp of MIPC}
$\mipc \vdash \varphi$ iff $\ms \vdash \varphi^t$ for each formula $\varphi$ of $\Lae$.
\end{theorem}

We next generalize the well-known notions of modal companion and intuitionistic fragment (see, e.g., \cite[Sec.~9.6]{CZ97}) to the monadic setting. 

\begin{definition} \cite[Def.~4.4]{BC24a} 
Let $\L$ be a monadic intuitionistic logic and $\M$ a monadic extension of $\ms$. We say that $\M$ is a \emph{modal companion} of $\L$ and that $\L$ is the \emph{intuitionistic fragment} of $\M$ provided
\[
\L\vdash\varphi \iff \M\vdash\varphi^t
\]
for every formula $\varphi$ of $\Lae$.
\end{definition}

To characterize modal companions of  intuitionistic modal logics, we need to recall the notion of morphism between descriptive $\mipc$-frames and between descriptive $\msfour$-frames. For this, we recall that a {\em p-morphism} between two Kripke frames $\mathfrak F_1 = (X_1,R_1)$ and $\mathfrak F_2 = (X_2,R_2)$ is a map $f \colon X_1 \to X_2$ satisfying $R_2[f(x)] = fR_1[x]$ for each $x \in X_1$. 

\begin{definition}\plabel{def:mipcfrm-morphisms}
\cite[Sec.~4]{Bez99} Let $\F_1=(X_1,R_1,Q_1)$ and $\F_2=(X_2,R_2,Q_2)$ be descriptive $\mipc$-frames. A map $f \colon X_1 \to X_2$ is a \emph{morphism of descriptive $\mipc$-frames} if 
\begin{enumerate}
\item\label[def:mipcfrm-morphisms]{def:mipcfrm-morphisms:item1} $f$ is continuous, 
\item\label[def:mipcfrm-morphisms]{def:mipcfrm-morphisms:item2} $f\colon (X_1,R_1)\to(X_2,R_2)$ is a p-morphism,
\item\label[def:mipcfrm-morphisms]{def:mipcfrm-morphisms:item3} $f\colon (X_1,Q_1)\to(X_2,Q_2)$ is a p-morphism,
\item\label[def:mipcfrm-morphisms]{def:mipcfrm-morphisms:item4} $Q_2^{-1}[f(x)]= R_2^{-1}fQ_1^{-1}[x]$ for each $x\in X_1$.
\end{enumerate}
Let $\mipcfrm$ denote the category of descriptive $\mipc$-frames and their morphisms. 
\end{definition}

\begin{remark}\label{rem:morph dmsfrm:item1}
The last condition of the above definition is equivalent to 
\[
E_{Q_2}[f(x)]= R_2^{-1}fE_{Q_1}[x] \mbox{ for each } x\in X_1
\]
(see \cite[Lem.~16]{Bez99}). However, it is important to emphasize that it is strictly weaker than saying that $f$ is a p-morphism with respect to $E_Q$
(see \cite[Ex.~5.16]{BC24a}).
\end{remark}

\begin{definition}\plabel{def:msfrm-morphisms}
\cite[Def.~3.8]{BC24a} Let $\G_1=(Y_1,R_1,E_1)$ and $\G_2=(Y_2,R_2,E_2)$ be descriptive $\ms$-frames. A map $f \colon Y_1 \to Y_2$ is a \emph{morphism of descriptive $\ms$-frames} if 
\begin{enumerate}
\item\label[def:msfrm-morphisms]{def:msfrm-morphisms:item1} $f$ is continuous, 
\item\label[def:msfrm-morphisms]{def:msfrm-morphisms:item2} $f\colon(Y_1,R_1)\to(Y_2,R_2)$ is a p-morphism,
\item\label[def:msfrm-morphisms]{def:msfrm-morphisms:item3} $f \colon (Y_1,E_1) \to (Y_2,E_2)$ is a p-morphism,
\end{enumerate}
Let $\msfrm$ denote the category of descriptive $\ms$-frames and their morphisms. 
\end{definition}

\begin{remark}
    We emphasize that the last condition of the above definition is stronger than the condition in \cref{rem:morph dmsfrm:item1}.
\end{remark}

We next connect descriptive $\ms$-frames with descriptive $\mipc$-frames. For this we recall the notion of the skeleton. 
Given a descriptive $\ms$-frame $\G=(Y,R,E)$ let $Q_E$ be the composite $E \circ R$; that is,
\begin{equation}\label{eq: QE}\tag{$\ast\ast$}
x \mathrel{Q_E} y \ \Longleftrightarrow \ \exists z \in Y : x \Rrel z \; \& \; z \Erel y.
\end{equation}

Since $R$ is reflexive, it is clear that $E \subseteq Q_E$ (but the converse is not always true). This will be used in \cref{lem:Casari and almost strong isotone implies isotone}.

\begin{definition}\ \plabel{def:skeleton}
\begin{enumerate}
    \item\label[def:skeleton]{def:skeleton:1} Define the {\em skeleton} of a descriptive $\ms$-frame $\G=(Y,R,E)$ to be the tuple $$\sk(\G)\coloneqq(X,R',Q')$$ where $X$
    is the quotient of $Y$ by the equivalence
relation $E_R$ on $Y$ induced by $R$ (see~\eqref{eq: ER}), $\pi \colon Y \to X$ is the quotient map,  
\[
\pi(x) \mathrel{R'} \pi(y) \iff x \Rrel y,
\]
and
\[
\pi(x) \mathrel{Q'} \pi(y) \iff x \mathrel{Q_E} y.
\]
\item\label[def:skeleton]{def:skeleton:2} If $\G_1=(Y_1,R_1,E_1)$, $\G_2=(Y_2,R_2,E_2)$, and $f \colon \G_1 \to \G_2$ is a $\msfrm$-morphism, we define $\sk(f) \colon \sk(\G_1) \to \sk(\G_2)$ by 
\[
\sk(f)(\pi_1(x))=\pi_2(f(x))
\]
for each $x \in Y_1$, where $\pi_1,\pi_2$ are the corresponding quotient maps.
\end{enumerate}
\end{definition}

\begin{lemma}\cite[Lem.~5.15]{BC24a}
$\sk \colon \msfrm \to \mipcfrm$ is a well-defined functor. \label{lem: skeleton}
\end{lemma}

For a monadic intuitionistic logic $\L$, let $\DF_\L$ be the full subcategory of $\DF_\mipc$ consisting of descriptive $\L$-frames; and for a monadic extension $\M$ of $\msfour$, define $\DF_\M$ similarly. 
Following \cite[p.~261]{CZ97}, we call an onto morphism $f \colon \F_1 \to \F_2$ a {\em reduction}. When there is a reduction from $\F_1$ to $\F_2$, we write $\F_1\twoheadrightarrow\F_2$. For a class $\K$ of descriptive $\mipc$ or $\msfour$-frames, let $$\R(\K) = \{ \F_2 \mid \F_1 \twoheadrightarrow \F_2 \mbox{ for some } \F_1 \in \K \}.$$ 
The next theorem characterizes all modal companions of a given monadic intuitionistic logic. 

\begin{theorem}{\cite[Thm.~5.12(2)]{BC24a}}\label{thm: auxiliary 2}
Let $\L$ be a monadic intuitionistic logic. A monadic extension $\M$ of $\msfour$ is a  modal companion of $\L$ iff 
$\DF_\L = \R (\sk[\DF_\M])$. 
\end{theorem}  

\begin{remark}
The proof of \cite[Thm.~5.12(2)]{BC24a} uses algebraic semantics, but the above reformulation is equivalent using the dual equivalence between the algebraic and descriptive frame semantics.
\end{remark}

\section{Esakia's theorem for \texorpdfstring{$\mpipc$}{M+IPC}}

We recall (see, e.g., \cite[p.~93]{CZ97}) that the {\em Grzegorczyk logic} $\Grz$ is the normal extension of $\sfour$ by the {\em Grzegorczyk axiom}
$$
\grz = 
\Box(\Box(p\to\Box p)\to p)\to p.
$$
Given a descriptive $\sfour$-frame $\G=(Y,R)$, we recall that $x\in Y$ is a {\em maximal point} of $U \subseteq Y$ provided $x \in U$ and $$(\forall y \in U)(x\mathrel{R}y \Longrightarrow x=y).$$ Let $\max U$ denote the set of maximal points of $U$. We have the following characterization of descriptive $\Grz$-frames:

\begin{theorem}\plabel{thm:dual characterization mgrz} \cite[p.~71]{Esa19} Let $\G=(Y,R)$ be a descriptive $\sfour$-frame.
\begin{enumerate}
    \item\label[thm:dual characterization mgrz]{thm:dual characterization mgrz:1} $\G$ is a descriptive $\Grz$-frame iff $U \subseteq R^{-1}[\max U]$ for each clopen $U \subseteq Y$. 
    \item\label[thm:dual characterization mgrz]{thm:dual characterization mgrz:2} If $R$ is a partial order, then $\G$ is a descriptive $\Grz$-frame.
    \item\label[thm:dual characterization mgrz]{thm:dual characterization mgrz:3} If $\G$ is finite, then $\G$ is a $\Grz$-frame iff $R$ is a partial order.
\end{enumerate}
\end{theorem}

As we pointed out in the introduction, Esakia \cite{Esa79} proved the following:

\begin{theorem} [Esakia's theorem]
    $\Grz$ is the largest modal companion of $\ipc$. Thus, the modal companions of $\ipc$ form the interval $[\sfour,\Grz]$ in the lattice of normal extensions of~$\sfour$. 
\end{theorem}

In order to explore Esakia's theorem in the monadic setting, we need to extend the functor~$\sigma$. However, as we pointed out in the introduction, $\sigma$ does not extend in general because if $\F=(X,R,Q)$ is a descriptive $\mipc$-frame, then $E_Q$ may not be a continuous relation, and hence $(X,R,E_Q)$ is not a descriptive $\ms$-frame (see, e.g., \cite[Rem~2.23]{BC24b}). However, it is clear that if $\F$ is finite, then $(X,R,E_Q)$ is an $\ms$-frame. We thus set:

\begin{definition}
    For a finite $\mipc$-frame $\F=(X,R,Q)$ let $\sigma\F=(X,R,E_Q)$,
    and for a morphism $f \colon \F_1 \to \F_2$ between finite $\mipc$-frames, let $\sigma \! f = f$.
\end{definition}

We use $\sigma$ and $\rho$ to obtain a relationship between finite $\mipc$-frames and finite $\mgrz$-frames. 

\begin{lemma}\label{lem: E=EQE}
    For each finite $\mgrz$-frame $\F=(X,R,E)$, we have $E=E_{Q_E}$. 
\end{lemma}

\begin{proof}
    Since $\F$ is a finite $\mgrz$-frame, $R$ is a partial order. Thus, \cite[Lem.~3(b)]{Bez99} applies, by which $E=E_{Q_E}$.
\end{proof}

\begin{proposition}\ \plabel{prop: sigma and rho finite}
    \begin{enumerate}
        \item\label[prop: sigma and rho finite]{prop: sigma and rho finite:1} For a finite $\mipc$-frame $\F=(X,R,Q)$, $\sigma\F$ is a finite $\mgrz$-frame and $\F\cong\rho\sigma\F$.
        \item\label[prop: sigma and rho finite]{prop: sigma and rho finite:2} For a finite $\mgrz$-frame $\G=(Y,R,E)$, $\rho\G$ is a finite $\mipc$-frame and $\G\cong\sigma\rho\G$.
    \end{enumerate}
\end{proposition}

\begin{proof}
    \eqref{prop: sigma and rho finite:1} Since $\F$ is finite, so is $\sigma\F$, and it follows from \cref{thm:dual characterization mgrz:3} that $\sigma\F$ is an $\mgrz$-frame. Moreover, since for all $x,y\in X$, $$x\mathrel{Q}y \iff \exists z \in X : x \Rrel z \; \& \; z \mathrel{E_Q}y,$$ the map $x \mapsto \{x\}$ is a bijection that preserves and reflects both $R$ and $Q$. Thus, it is an isomorphism of the $\mipc$-frames $\F$ and $\rho\sigma\F$. 

\eqref{prop: sigma and rho finite:2} Clearly $\rho\G$ is finite, and it is an $\mipc$-frame by \cref{lem: skeleton}. By \cref{thm:dual characterization mgrz:3}, $R$ is a partial order. 
Therefore, 
\cref{lem: E=EQE}
yields
that the map $x \mapsto \{x\}$ is a bijection that preserves and reflects both $R$ and $E$. Thus, it is an isomorphism of the $\mgrz$-frames 
$\G$ and~$\sigma\rho\G$.
\end{proof}

\begin{definition}
    Let $\mipcfin$ denote the category of finite $\mipc$-frames and their morphisms, and $\mgrzfin$ the category of finite $\mgrz$-frames and their morphisms.     
\end{definition}

In view of \cref{prop: sigma and rho finite}, one might expect that $\rho$ and $\sigma$ establish an equivalence of $\mgrzfin$ and $\mipcfin$. However, this is not the case because there exist
$\F_1,\F_2 \in \mipcfin$ and a $\mipcfin$-morphism $f$ between them such that $\sigma \! f\colon\sigma\F_1\to\sigma\F_2$ is not a p-morphism with respect to $E_Q$, and hence $\sigma$ is not well-defined on morphisms: 

\begin{example}
Let $\F_1=(X_1,R_1,Q_1)$ and $\F_2=(X_2,R_2,Q_2)$ be the finite $\mipc$-frames shown below, where the black arrows indicate the partial orders $R_1,R_2$ and the red circles 
the equivalence relations $E_{Q_1},E_{Q_2}$. 

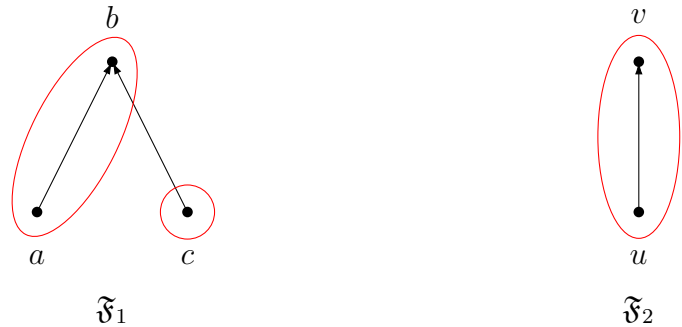
\begin{figure}[ht]
\begin{tikzpicture}[-{Latex[width=1mm]}]
\coordinate (3BL) at (-5,-5);
\coordinate (3BR) at (-3,-5);
\coordinate (3T) at (-4,-3);
\fill (3BL) circle(2pt);
\fill (3BR) circle(2pt);
\fill (3T) circle(2pt);
\draw (3BL) -- (3T);
\draw (3BR) -- (3T);
\clustertwo{3BL}{3T}{1.3}{1.25};
\clusterone{3BR}{1.2};
\node at ([shift={(-90:0.6)}]3BL) {$a$};
\node at ([shift={(90:0.6)}]3T) {$b$};
\node at ([shift={(-90:0.6)}]3BR) {$c$};
\node at (-4,-6.3) {$\F_1$};
\coordinate (4B) at (3,-5);
\coordinate (4T) at (3,-3);
\fill (4B) circle(2pt);
\fill (4T) circle(2pt);
\draw (4B) -- (4T);
\clustertwo{4B}{4T}{1.35}{1.2};
\node at ([shift={(-90:0.6)}]4B) {$u$};
\node at ([shift={(90:0.6)}]4T) {$v$};
\node at (3,-6.3) {$\F_2$};
\end{tikzpicture}
\end{figure}
We have $\sigma\F_i=(X_i,R_i,E_{Q_i})$ for $i=1,2$.  
Define $f\colon X_1\to X_2$ by 
$f(a)=u$ and ${f(b)=f(c)=v}$.
Then 
$f$ is not a p-morphism with respect to $E_Q$ since 
\[
E_{Q_2}[f(c)] = \{u,v\} \neq \{v \} = fE_{Q_1}[c].
\]
Therefore, $\sigma \! f=f$
is not a $\mgrzfin$-morphism.
On the other hand, $f$ is a $\mipcfin$-morphism because $E_{Q_2}[f(x)] = fE_{Q_1}[x]$ for $x=a,b$ and 
\[
E_{Q_2}[f(c)] = \{u,v\} = R_2^{-1}fE_{Q_1}[c].
\] 
\end{example}

By \cref{thm: auxiliary 2}, if $\M$ is a modal companion of $\mipc$, then for each finite $\mipc$-frame $\F$ there is a descriptive $\M$-frame $\G$ such that $\rho\G\twoheadrightarrow \F$. The main obstacle in proving Esakia's theorem for $\mipc$ is that we no longer have $\G\twoheadrightarrow\sigma\F$, which results in the following:

\begin{theorem} \cite[Thm.~5.10]{BC24b}
$\mipc$ has no largest modal companion. 
\end{theorem}

The situation changes when we work with Esakia's amended predicate intuitionistic logic \cite{Esa98} and its monadic fragment $\mpipc$.

\begin{definition} \cite[p.~429]{BBI23}
The {\em amended monadic intuitionistic calculus} $\mpipc$ is obtained by adding to $\mipc$ 
the {\em monadic Casari formula} $$\forall((p \to \forall p) \to \forall p) \to \forall p,$$ and the {\em amended monadic Grzegorczyk logic} $\mpgrz$ by adding to $\mgrz$ the G\"odel translation of the monadic Casari formula.
\end{definition}

\begin{proposition}\ \plabel{prop: clean clusters}
\begin{enumerate}
\item\label[prop: clean clusters]{prop: clean clusters: 1}\cite[Thm.~5.17]{BBI23} $\mpipc$ has the fmp.
\item\label[prop: clean clusters]{prop: clean clusters: 2}\cite[Lem.~4.4]{BBI23} A finite $\mipc$-frame $\F=(X,R,Q)$ is an $\mpipc$-frame 
iff 
\[
(\forall x,y \in X)(x\Rrel y \mbox{ and } x\mathrel{E_Q}y \Longrightarrow x=y).
\]
\end{enumerate}
\end{proposition}

\needspace{5\baselineskip} 
\begin{remark}\
\begin{enumerate}
    \item The condition in \cref{prop: clean clusters: 2}
    is known as having clean clusters (see \cite[Def.~3.6]{BBI23}).
    \item \cref{prop: clean clusters: 2} has a generalization to all descriptive $\mipc$-frames (see \cite[Lem.~38]{Bez00} and \cite[Lem.~4.2]{BBI23}). For our purposes, it is enough to work with finite $\mpipc$-frames.
    \item For a characterization of descriptive $\mpgrz$-frames see \cite[Lem.~4.8]{BBI23}.
\end{enumerate}
\end{remark}

\begin{definition}
    Let $\mpipcfin$ denote the full subcategory of $\mipcfin$ consisting of $\mpipc$-frames, and $\mpgrzfin$ the full subcategory of $\mgrzfin$ consisting of $\mpgrz$-frames.     
\end{definition}

\begin{theorem}
    The functors $\rho\colon\mpgrzfin\to\mpipcfin$ and $\sigma\colon\mpipcfin\to\mpgrzfin$ yield an equivalence of $\mpgrzfin$ and $\mpipcfin$.
\end{theorem}

\begin{proof}
    In view of \cref{prop: sigma and rho finite}, it is sufficient to show that $\sigma$ is well defined on $\mpipcfin$-morphisms. Let $\F_1=(X_1,R_1,Q_1)$ and $\F_2=(X_2,R_2,Q_2)$ be finite $\mpipc$-frames and $f\colon\F_1\to\F_2$ a $\mpipcfin$-morphism. It is enough to show that $E_{Q_2}[f(x)]\subseteq fE_{Q_1}[x]$ for all $x \in X_1$. Let $f(x) \mathrel{E_{Q_2}} y$. By \cref{rem:morph dmsfrm:item1}, there is $z \in X_1$ such that $x\mathrel{E_{Q_1}}z$ and $y \Rrel_2 f(z)$. From $x \mathrel{E_{Q_1}} z$ it follows that $f(x) \Erel_{Q_2} f(z)$. Therefore, $y \Rrel_2 f(z)$ and $y \Erel_{Q_2} f(z)$. Since $\F_2$ is an $\mpipc$-frame, $y=f(z)$ by \cref{prop: clean clusters: 2}, concluding the proof.   
\end{proof}

We now prove the main result of this paper, that the modal companions of $\mpipc$ form the interval $[\mpsfour,\mpgrz]$ in the lattice of monadic extensions of $\msfour$, where $\mpsfour$ is the monadic extension of $\msfour$ by the G\"odel translation of the monadic Casari formula. For this we utilize the following:

\begin{theorem}\ \plabel{thm: mgrz properties}
\begin{enumerate}
    \item {\cite[Thm.~4.12]{BBI23}} \label[thm: mgrz properties]{thm: modal companion} $\mpgrz$ is a modal companion of $\mpipc$.
    \item {\cite[Thm.~6.16]{BBI23}} \label[thm: mgrz properties]{thm: auxiliary 1} $\mpgrz$ has the fmp. 
\end{enumerate}
\end{theorem}

\begin{theorem} \label{thm: main}
$\mpgrz$ is the greatest modal companion of $\mpipc$.
\end{theorem}

\begin{proof}
Let $\M$ be a modal companion of $\mpipc$.  By \cref{thm: auxiliary 1}, $\mpgrz$ has the fmp. Therefore, to show that $\M \subseteq \mpgrz$, it is enough to show that each finite $\mpgrz$-frame $\G$ is an $\M$-frame. By Theorems \ref{thm: auxiliary 2} and \ref{thm: mgrz properties}\eqref{thm: modal companion}, $\DF_{\mpipc}=\R(\sk[\DF_{\mpgrz}])$. Therefore, since $\sk(\G) \in \sk[\DF_{\mpgrz}] \subseteq \R(\sk[\DF_{\mpgrz}])$, we obtain that $\sk(\G) \in \DF_{\mpipc}$.
Thus, since $\G$ is finite, $\rho\G$ is a finite $\mpipc$-frame.  
Applying \cref{thm: auxiliary 2} to $\M$  
yields a descriptive $\M$-frame $\H$ such that $\rho\G \in \R(\rho\H)$. 

\begin{claim}\label{lem:Casari and almost strong isotone implies isotone}
Let $\F=(X,R,Q)$ be a finite $\mpipc$-frame and $\H=(Y,S,E)$ a descriptive $\msfour$-frame. 
If $\rho\H\twoheadrightarrow\F$, then $\H\twoheadrightarrow\sigma\F$. 
\end{claim}

\begin{proofclaim}
Suppose that 
$f \colon \rho\H \to \F$ is a reduction in $\mipcfrm$. Let $\pi \colon \H \to \rho\H$ be the quotient map. Define $g \colon \H \to \sigma\F$ by $g(y)=f\pi(y)$ for each $y\in Y$. Since $\sigma\F=(X,R,E_Q)$, it is clear that $g$ is a well-defined continuous onto map.
Because $\pi\colon(Y,S)\to(Y',S')$ and $f\colon(Y',S')\to(X,R)$ are p-morphisms, so is $g\colon(Y,S)\to(X,R)$. Thus, it is left to show that $g\colon(Y,E)\to(X,E_Q)$ is a p-morphism. 

First, suppose that $x,y \in Y$ with $x \Erel y$. Then $x \mathrel{Q_E} y$ and $y \mathrel{Q_E} x$ (see \eqref{eq: QE} for the definition of $Q_E$),
so $\pi(x) \mathrel{Q'} \pi(y)$ and $\pi(y) \mathrel{Q'} \pi(x)$ by \cref{def:skeleton:1}. Therefore, $f\pi(x) \Qrel f\pi(y)$ and $f\pi(y) \Qrel f\pi(x)$, yielding that $g(x) \mathrel{E_Q} g(y)$ (see \eqref{eq: ER} for the definition of~$E_Q$). 

Next, suppose that $x\in X$, $y\in Y$, and $g(y) \mathrel{E_Q} x$. Then $f\pi(y) \mathrel{E_Q} x$. Since $f$ is a morphism of descriptive $\mipc$-frames, there is $z \in Y$ such that $\pi(y) \mathrel{E_{Q'}} \pi(z)$ and $x \Rrel f\pi(z)$ (see \cref{rem:morph dmsfrm:item1}). 
The former implies that $\pi(y) \mathrel{Q'} \pi(z)$ and $\pi(z) \mathrel{Q'} \pi(y)$. 
From $\pi(z) \mathrel{Q'} \pi(y)$ it follows that $z \mathrel{Q_E} y$ 
(see \cref{def:skeleton:1}),
so there is $u \in Y$ such that $z \mathrel{S} u$ and $u\Erel y$. 
From $z \mathrel{S} u$ it follows that $f\pi(z) \Rrel f\pi(u)$, which together with $x \Rrel f\pi(z)$ gives that $x \Rrel f\pi(u)$, so $x \Rrel g(u)$. Also, $u \Erel y$ implies that 
$g(u) \mathrel{E_Q} g(y)$
(see the previous paragraph). 
The latter together with $g(y) \mathrel{E_Q} x$ yields that $x \mathrel{E_Q} g(u)$.
Since $\F$ is a finite $\mpipc$-frame, from $x \Rrel g(u)$ and $x \mathrel{E_Q} g(u)$ it follows that  
$x=g(u)$ (see \cref{prop: clean clusters: 2}). Thus, there is $u\in Y$ such that $y \mathrel{E} u$ and $g(u)=x$, and hence $f\colon(Y,E)\to(X,E_Q)$ is a p-morphism, concluding the proof. 
\end{proofclaim}

As an immediate consequence of \cref{lem:Casari and almost strong isotone implies isotone}, we obtain that $\sigma\rho\G \in \R(\H)$. This implies that ${\sigma\rho\G \in \DF_\M}$ because $\DF_\M$ is closed under $\R$. Thus, since $\sigma\rho\G \cong \G$, we conclude that $\G$ is an $\M$-frame. 
\end{proof}

To prove that the modal companions of $\mpipc$ form the interval $[\mpsfour,\mpgrz]$ in the lattice of monadic extensions of $\msfour$, it is left to show that $\mpsfour$ is the least modal companion of $\mpipc$. For this, given a monadic intuitionistic logic $\L$, let 
\[
\tau(\L) = \ms + \{\varphi^t : \L \vdash \varphi \}.
\] 

\begin{proposition} \ \plabel{prop:Kripke complete admit modal comp}
\begin{enumerate}
\item\cite[Prop.~3.9]{BC24b} \label[prop:Kripke complete admit modal comp]{prop:Kripke complete admit modal comp:1} Let $\Gamma$ be a set of formulas in 
$\Lae$. Then 
\[
\tau(\mipc+\Gamma)=\ms+\{\gamma^t : \gamma \in \Gamma\}.
\]
\item\cite[Prop.~3.11]{BC24b} \label[prop:Kripke complete admit modal comp]{prop:Kripke complete admit modal comp:2} If $\L$ is a Kripke complete monadic intuitionistic logic, then $\tau (\L)$ is a modal companion of $\L$.
\end{enumerate}
\end{proposition}

\begin{remark}
    We emphasize that the assumption in \cref{prop:Kripke complete admit modal comp:2} that $\L$ is Kripke complete is essential, and that it remains open whether $\tau(\L)$ is a modal companion of an arbitrary monadic intuitionistic logic $\L$ (see \cite[Rem.~3.10]{BC24b}).
\end{remark}

Putting \cref{thm: main,prop:Kripke complete admit modal comp} together yields:

\begin{theorem}
The modal companions of $\mpipc$ form the interval $[\mpsfour,\mpgrz]$ in the lattice of monadic extensions of $\msfour$.
\end{theorem}

\begin{proof}
    By \cref{prop:Kripke complete admit modal comp:1},  $\mpsfour=\tau(\mpipc)$. Therefore, by Propositions~\ref{prop: clean clusters}\eqref{prop: clean clusters: 1} and~\ref{prop:Kripke complete admit modal comp}\eqref{prop:Kripke complete admit modal comp:2}, $\mpsfour$ is a modal companion of $\mpipc$. Let $\M$ be a modal companion of $\mpipc$. By \cref{thm: main}, $\M\subseteq\mpgrz$. Also, since $\M$ proves the G\"odel translation of Casari's formula, $\mpsfour\subseteq\M$, concluding the proof.
\end{proof}

\begin{remark}
    The algebraic semantics for $\mipc$ is given by monadic Heyting algebras \cite{Bez98}, and that for $\msfour$ by monadic $\sfour$-algebras \cite{BBI23,BC24a}. Using the dual equivalence between the categories of monadic Heyting algebras and descriptive $\mipc$-frames \cite{Bez99} and that between the categories of monadic $\sfour$-algebras and descriptive $\msfour$-frames \cite{BC24a}, the above result can be formulated as follows: A variety of monadic $\sfour$-algebras is the variety corresponding to a modal companion of $\mpipc$ iff it contains the variety of $\mpgrz$-algebras and is contained in the variety of $\mpsfour$-algebras.
\end{remark}


\begin{thebibliography}{GKWZ03}

\bibitem[BBI23]{BBI23}
G.~Bezhanishvili, K.~Brantley, and J.~Ilin.
\newblock Monadic intuitionistic and modal logics admitting provability interpretations.
\newblock {\em J. Symb. Log.}, 88(1):427--467, 2023.

\bibitem[BC23]{BC23a}
G.~Bezhanishvili and L.~Carai.
\newblock Temporal interpretation of monadic intuitionistic quantifiers.
\newblock {\em Rev. Symb. Log.}, 16(1):164--187, 2023.

\bibitem[BC24]{BC24b}
G.~Bezhanishvili and L.~Carai.
\newblock Failure of {E}sakia's theorem in the monadic setting.
\newblock Submitted. Available at ar{X}iv:2409.05607, 2024.

\bibitem[BC25]{BC24a}
G.~Bezhanishvili and L.~Carai.
\newblock Failure of the {B}lok-{E}sakia theorem in the monadic setting.
\newblock {\em Ann. Pure Appl. Logic}, 176(4):Paper No. 103527, 2025.

\bibitem[Bez98]{Bez98}
G.~Bezhanishvili.
\newblock Varieties of monadic {H}eyting algebras. {I}.
\newblock {\em Studia Logica}, 61(3):367--402, 1998.

\bibitem[Bez99]{Bez99}
G.~Bezhanishvili.
\newblock Varieties of monadic {H}eyting algebras. {II}. {D}uality theory.
\newblock {\em Studia Logica}, 62(1):21--48, 1999.

\bibitem[Bez00]{Bez00}
G.~Bezhanishvili.
\newblock Varieties of monadic {H}eyting algebras. {III}.
\newblock {\em Studia Logica}, 64(2):215--256, 2000.

\bibitem[BK24]{BK24}
G.~Bezhanishvili and M.~Khan.
\newblock The monadic {G}rzegorczyk logic.
\newblock Submitted. Available at ar{X}iv:2412.10854, 2024.

\bibitem[Bul66]{Bul66}
R.~A. Bull.
\newblock {MIPC} as the formalisation of an intuitionist concept of modality.
\newblock {\em J. Symb. Log.}, 31(4):609--616, 1966.

\bibitem[CZ97]{CZ97}
A.~Chagrov and M.~Zakharyaschev.
\newblock {\em Modal logic}, volume~35 of {\em Oxford Logic Guides}.
\newblock The Clarendon Press, Oxford University Press, New York, 1997.

\bibitem[Esa79]{Esa79}
L.~Esakia.
\newblock On the variety of {G}rzegorczyk algebras.
\newblock In {\em Studies in nonclassical logics and set theory ({R}ussian)}, pages 257--287. ``Nauka'', Moscow, 1979.

\bibitem[Esa98]{Esa98}
L.~Esakia.
\newblock Quantification in intuitionistic logic with provability smack.
\newblock {\em Bull. Sect. Logic Univ. Lódź}, 27:26--28, 1998.

\bibitem[Esa19]{Esa19}
L.~Esakia.
\newblock {\em Heyting algebras. {D}uality theory}, volume~50 of {\em Trends in Logic. Translated from the Russian by A. Evseev. Edited by G. Bezhanishvili and W. Holliday.}
\newblock Springer, 2019.

\bibitem[FS77]{FS77}
G.~Fischer~Servi.
\newblock On modal logic with an intuitionistic base.
\newblock {\em Studia Logica}, 36(3):141--149, 1977.

\bibitem[FS78]{FS78a}
G.~Fischer~Servi.
\newblock The finite model property for {MIPQ} and some consequences.
\newblock {\em Notre Dame J. Formal Logic}, 19(4):687--692, 1978.

\bibitem[GKWZ03]{GKWZ03}
D.~M. Gabbay, A.~Kurucz, F.~Wolter, and M.~Zakharyaschev.
\newblock {\em Many-dimensional modal logics: theory and applications}, volume 148 of {\em Studies in Logic and the Foundations of Mathematics}.
\newblock North-Holland Publishing Co., Amsterdam, 2003.

\bibitem[Nau91]{Nau91}
P.~G. Naumov.
\newblock Modal logics that are conservative over intuitionistic predicate calculus.
\newblock {\em Moscow. Univ. Math. Bull. ({R}ussian)}, 6:86--90, 1991.

\bibitem[Ono77]{Ono77}
H.~Ono.
\newblock On some intuitionistic modal logics.
\newblock {\em Publ. Res. Inst. Math. Sci.}, 13(3):687--722, 1977.

\bibitem[Ono87]{Ono87}
H.~Ono.
\newblock Some problems in intermediate predicate logics.
\newblock {\em Rep. Math. Logic}, 21:55--67, 1987.

\bibitem[OS88]{OS88}
H.~Ono and N.-Y. Suzuki.
\newblock Relations between intuitionistic modal logics and intermediate predicate logics.
\newblock {\em Rep. Math. Logic}, 22:65--87, 1988.

\bibitem[Pri57]{Pri57}
A.~N. Prior.
\newblock {\em Time and modality}.
\newblock The Clarendon Press, Oxford University Press, New York, 1957.

\end{thebibliography}
\end{document}